\documentclass[11pt,reqno]{amsart}

\usepackage{mathrsfs}

\usepackage{hyperref}

\usepackage{graphicx}
\usepackage{xcolor}
\usepackage{enumerate}
\usepackage{bm}

\newtheorem{theorem}{Theorem}[section]
\newtheorem{lemma}[theorem]{Lemma}

\theoremstyle{definition}

\theoremstyle{remark}

\numberwithin{equation}{section}

\begin{document}

\title{The boundary case of the $J$-flow
%\thanks{}
}

\author{Wei Sun}

\address{Institute of Mathematical Sciences, ShanghaiTech University, Shanghai, China}
\email{sunwei@shanghaitech.edu.cn}

%\date{}

\begin{abstract}
In this paper, we shall study the boundary case for the $J$-flow under certain geometric assumptions. 

\end{abstract}

\maketitle

%\setcounter{section}{-1}   

%************************************************

\medskip
\section{Introduction}

Let $M$ be a compact manifold without boundary of complex dimension $n\geq 2$, carrying two K\"ahler forms $\chi$ and $\omega$. 
The $J$-flow is
\begin{equation}
\label{J-flow}
	\frac{\partial\varphi}{\partial t} = c - \frac{n (\chi + \sqrt{-1}\partial\bar\partial\varphi)^{n - 1} \wedge \omega}{(\chi + \sqrt{-1} \partial\bar\partial\varphi)^n} ,\qquad \varphi (z, 0) = \varphi_0 (z), 
\end{equation}
where $\chi + \sqrt{-1}\partial\bar\partial \varphi_0 > 0$ and
\begin{equation*}
	c := \frac{n \int_M \chi^{n - 1} \wedge \omega}{\int_M \chi^n}. 
\end{equation*}
The $J$-flow was constructed to study
\begin{equation}
\label{equation-1-2}
	c \left(\chi + \sqrt{-1}\partial\bar\partial u\right)^n = n \left(\chi + \sqrt{-1}\partial\bar\partial u\right)^{n - 1} \wedge \omega, \qquad \chi + \sqrt{-1}\partial\bar\partial u > 0 ,
\end{equation}
which is a stationary state of the $J$-flow~\eqref{J-flow}. 
Equation~\eqref{equation-1-2} was independently discovered by Donaldson~\cite{Donaldson1999} and Chen~\cite{Chen2000} under different geometric backgrounds.

For smooth admissible solutions, Donaldson~\cite{Donaldson1999} proposed that it is a sufficient condition that
\begin{equation}
\label{condition-1-3}
	[nc \chi - \omega] > 0,
\end{equation}
which was confirmed by Chen~\cite{Chen2000} for K\"ahler surfaces. 
Indeed, Equation~\eqref{equation-1-2} can be rewritten as a complex Monge-Amp\`ere equation on K\"ahler surfaces, which was solved by Yau~\cite{Yau1978}. 
However, Condition~\eqref{condition-1-3} does not work well in higher dimensions. 
Chen~\cite{Chen2004} proved the long time smoothness and existence of the $J$-flow. 
Later, a new condition was proposed by  Song and Weinkove~\cite{SongWeinkove2008}, that is,
there exists a K\"ahler form $\chi' \in [\chi]$ such that
\begin{equation}
\label{inequality-1-4}
c \chi'^{n - 1} - (n - 1) \chi'^{n - 2} \wedge \omega > 0 .
\end{equation}
Song and Weinkove~\cite{SongWeinkove2008} proved that Condition~\eqref{inequality-1-4} is sufficient and necessary for the existence of smooth admissible solution to Equation~\eqref{equation-1-2} by proving the convergence for the $J$-flow~\eqref{J-flow}. 
Later, Fang, Lai and Ma~\cite{FangLaiMa2011} extended the results to complex Monge-Amp\`ere type equations by parabolic flows, and defined the sufficient and necessary solvability condition as {\it cone condition}.

When Condition~\eqref{inequality-1-4} degenerates to the boundary case
\begin{equation}
\label{boundary-case-1-5}
c \chi'^{n - 1} - (n - 1) \chi'^{n - 2} \wedge \omega \geq 0 ,
\end{equation}
it is interesting to study the behavior of the $J$-flow~\eqref{J-flow}. 
On K\"ahler surfaces, Fang, Lai, Song and Weinkove~\cite{FangLaiSongWeinkove2014} adapted the trick of Chen~\cite{Chen2000}, and solved
\begin{equation}
\label{equation-1-6}
	  (c \chi - \omega + \sqrt{-1} \partial\bar\partial u)^2 = \omega^2 .
\end{equation}
According to  \eqref{boundary-case-1-5}, $[c\chi - \omega]$ is semipositive and big. 
Eyssidieux, Guedj and Zeriahi~\cite{EyssidieuxGuedjZeriahi2009} showed that there is a bounded pluripotential solution to Equation~\eqref{equation-1-6}, which is smooth in $Amp (c\chi - \omega)$.
Taking advantage of the $L^\infty$ estimate for \eqref{equation-1-6}, Fang, Lai, Song and Weinkove~\cite{FangLaiSongWeinkove2014} proved the $t$-independent $C^0$ estimate, and then studied the behavior of the solution. In this paper, we also derive the $C^0$ estimate via this way.

%\newpage

To study the boundary case, we need to impose more geometric conditions. 
Instead of Equation~\eqref{J-flow}, we study more general gradient flows,
\begin{equation}
\label{equation-3-1}
	\frac{\partial\varphi}{\partial t} = c - \frac{n (\chi + \tilde \chi + \sqrt{-1} \partial\bar\partial\varphi)^{m } \wedge \omega^{n - m}}{(\chi + \tilde \chi + \sqrt{-1} \partial\bar\partial\varphi)^n} , \qquad \varphi (z,0) = \varphi_0 (z) \in \mathcal{H},
\end{equation}
where $1 \leq m < n$, 
$
	\mathcal{H} := \{ u \in C^\infty (M) | \chi + \tilde \chi + \sqrt{-1} \partial\bar\partial u > 0\}
$
 and $c$ is defined by
\begin{equation}
	c := \frac{n \int_M (\chi + \tilde \chi)^{m } \wedge \omega^{n - m}}{\int_M (\chi + \tilde \chi)^n} >  0 .
\end{equation}
We impose the condition that $\tilde \chi$ is semipositve and big, and the boundary case of cone condition
\begin{equation}
\label{inequality-1-5}
	c \chi^{n - 1} - m \chi^{m -1} \wedge \omega^{n - m} \geq 0 .
\end{equation}
Then the stationary state of \eqref{equation-3-1} is a Monge-Amp\`ere type equation, 
\begin{equation}
\label{elliptic-equation}
c (\chi + \tilde \chi + \sqrt{-1} \partial\bar\partial \psi)^n = n (\chi + \tilde \chi +\sqrt{-1} \partial\bar\partial \psi)^m \wedge \omega^{n - m}.
\end{equation}
With the assumptions above, we can prove the following theorem.
\begin{theorem}
\label{main-theorem}
Let $(M,\omega)$ be a compact complex manifold without boundary of complex dimension $n \geq 2$. Assume that there is a K\"ahler form $\chi$ and a big semipositive form $\tilde \chi$ satisfying the boundary case~\eqref{inequality-1-5} of cone condition. Then for any $\phi_0 \in \mathcal{H}$, the solution $\varphi(t)$ to Equation~\eqref{equation-3-1} converges in $C^\infty_{loc} (Amp (\tilde \chi))$ to a  function $\varphi_{\infty}$, which is a pluripotential solution to Equation~\eqref{elliptic-equation} such that 
\begin{equation*}
	 \sum^n_{i = 0} \int_M \varphi_\infty \left(\chi + \tilde \chi + \sqrt{-1} \partial\bar\partial \varphi_\infty\right)^i \wedge (\chi + \tilde \chi)^{n - i} 
	=
	 \sum^n_{i = 0} \int_M \varphi_0 \left(\chi + \tilde \chi + \sqrt{-1} \partial\bar\partial \varphi_0\right)^i \wedge (\chi + \tilde \chi)^{n - i}
	. 
\end{equation*}

\end{theorem}

\medskip

\section{Preliminary}

In this section, we shall recall some fundamental notations and formulae. 
In this paper, we denote by the letter $C$ a positive constant number, which vary from line to line.

\medskip

\subsection{The formulae in a local chart}

Rewriting Equation~\eqref{equation-3-1},
\begin{equation}
\label{equation-3-12}
	\frac{\partial \varphi}{\partial t} = c - \frac{n}{C^m_n} S_{n - m} (X^{-1}) .
\end{equation}
Differentiating \eqref{equation-3-12},
\begin{equation}
\label{equation-3-13}
	\partial_t \left(\partial_t\varphi\right)
	=
	\frac{n}{C^m_n} \sum_i S_{n - m - 1;i}(X^{-1}) (X^{i\bar i})^2 \left( \partial_t \varphi \right)_{i\bar i} ,
\end{equation}
\begin{equation}
	\partial_t \varphi_l 
	=
	\frac{n}{C^m_n} \sum_i S_{n - m - 1;i}(X^{-1}) (X^{i\bar i})^2 X_{i\bar il} ,
\end{equation}
and
\begin{equation}
\label{inequality-3-15}
\begin{aligned}
	\partial_t \varphi_{l\bar l}
%	&\leq 
%	 \frac{n}{C^m_n} \sum_i S_{n - m - 1;i}(X^{-1}) (X^{i\bar i})^2 X_{i\bar il\bar l} \\
%	 &\qquad - \frac{n}{C^m_n}\sum_{i,j} S_{n - m - 1;i}(X^{-1}) (X^{i\bar i})^2 X^{j\bar j} X_{j\bar i\bar l } X_{i\bar jl} \\
%	 &\qquad - \frac{n}{C^m_n} \sum_{i,j} \frac{S_{n - m - 1;i} (X^{-1}) S_{n - m - 1;j} (X^{-1})}{S_{n - m }(X^{-1})} (X^{i\bar i})^2 (X^{j\bar j})^2 X_{i\bar il} X_{j\bar j\bar l} \\
%	 &=
%	 	 \frac{n}{C^m_n} \sum_i S_{n - m - 1;i}(X^{-1}) (X^{i\bar i})^2 X_{i\bar il\bar l} \\
%	 	 &\qquad - \frac{n}{C^m_n}\sum_{i,j} S_{n - m - 1;i}(X^{-1}) (X^{i\bar i})^2 X^{j\bar j} X_{j\bar i\bar l } X_{i\bar jl} + \frac{1}{\varphi_t - c} \partial_t \varphi_l \partial_t \varphi_{\bar l} \\
	&\leq 
	 	 \frac{n}{C^m_n} \sum_i S_{n - m - 1;i}(X^{-1}) (X^{i\bar i})^2 X_{i\bar il\bar l} \\
	 	 &\qquad - \frac{n}{C^m_n}\sum_{i,j} S_{n - m - 1;i}(X^{-1}) (X^{i\bar i})^2 X^{j\bar j} X_{j\bar i\bar l } X_{i\bar jl} 
	 . 
\end{aligned}
\end{equation}
Applying maximum principle to Equation~\eqref{equation-3-13}, we can see that $\partial_t \varphi$ reaches the extremal values at $t = 0$, i.e. 
\begin{equation}
\label{inequality-2-5}
	 \inf_{M \times \{0\}} \frac{\partial\varphi}{\partial t} \leq \frac{\partial \varphi}{\partial t} \leq \sup_{M \times \{0\}} \frac{\partial\varphi}{\partial t} ,
\end{equation}
on the maximal time interval $[0,T)$. 
Consequently
\begin{equation}
\label{inequality-2-6}
\begin{aligned}
	\inf_M  \frac{(\chi + \tilde \chi + \sqrt{-1} \partial\bar\partial\varphi_0)^n}{(\chi + \tilde \chi + \sqrt{-1} \partial\bar\partial \varphi_0)^m \wedge \omega^{n - m}} 
	&\leq \frac{(\chi + \tilde \chi + \sqrt{-1} \partial\bar\partial\varphi)^n}{(\chi + \tilde \chi + \sqrt{-1} \partial\bar\partial \varphi)^m \wedge \omega^{n - m}} \\
	&\leq \sup_M  \frac{(\chi + \tilde \chi + \sqrt{-1} \partial\bar\partial\varphi_0 )^n}{(\chi + \tilde \chi + \sqrt{-1} \partial\bar\partial \varphi_0)^m \wedge \omega^{n - m}} .
\end{aligned}
\end{equation}
Therefore, $\chi + \tilde \chi + \sqrt{-1} \partial\bar\partial \varphi$ remains K\"ahler at any time on the maximal time interval $[0,T)$.

\medskip
\subsection{$J$-functionals}

As in \cite{Chen2000}\cite{FangLaiMa2011}\cite{Sun2015p}, we need to use a series of $J$-functional. 
For any curve $v(s) \in \mathcal{H}$, we define the functional $J_i (\chi + \tilde \chi , v)$ $(0 \leq i \leq n)$ by
\begin{equation*}
	\frac{d J_i}{d s} := \int_M \frac{\partial v}{\partial s} (\chi + \tilde \chi + \sqrt{-1} \partial\bar\partial v)^i \wedge \omega^{n - i}.
\end{equation*}
Then we obtain a formula for $J_i$ with
\begin{equation*}
	J_i (\chi + \tilde \chi, u) 
	= 
	\int^1_0 \left(\int_M \frac{\partial v}{\partial s} (\chi + \tilde \chi + \sqrt{-1} \partial\bar\partial v)^i \wedge \omega^{n - i} \right)ds
	,
\end{equation*}
for any path $v(s) \in \mathcal{H}$ connecting $0$ and $u$. 
Since $\mathcal{H}$ is simply connected, the functionals are independent of the choice of path. 
Along the straight line $v(s) = su$, we obtain that
\begin{equation}
\label{equality-3-11}
\begin{aligned}
	J_i (\chi + \tilde \chi, u)
	&=
	\int^1_0 \left( \int_M u (\chi + \tilde \chi + \sqrt{-1}\partial\bar\partial v)^i \wedge \omega^{n - i}\right) ds \\
%	&= 
%	\int^1_0 \left( \int_M u \sum^i_{j = 0} C^j_i s^j (1 - s)^{i - j} (\chi + \tilde \chi + \sqrt{-1} \partial\bar\partial u)^j \wedge (\chi + \tilde \chi )^{i - j} \wedge \omega^{n - i}\right)ds \\
%	&=
%	\sum^i_{j = 0} C^j_i \int^1_0 s^j (1 - s)^{i - j} ds \int_M u (\chi + \tilde \chi + \sqrt{-1 }\partial\bar\partial u)^j \wedge (\chi + \tilde\chi)^{i - j} \wedge \omega^{n - i} \\
%	&=
%	\sum^i_{j = 0} C^j_i B(j + 1, i - j + 1) \int_M u (\chi + \tilde \chi + \sqrt{-1 }\partial\bar\partial u)^j \wedge (\chi + \tilde\chi)^{i - j} \wedge \omega^{n - i} \\
%	&=
%	\sum^i_{j = 0} C^j_i \frac{\Gamma(j + 1) \Gamma(i - j + 1)}{\Gamma(i + 2)} \int_M u (\chi + \tilde \chi + \sqrt{-1 }\partial\bar\partial u)^j \wedge (\chi + \tilde\chi)^{i - j} \wedge \omega^{n - i} \\
%	&= \sum^i_{j = 0} \frac{i!}{j!(i - j)!} \frac{j! (i - j)!}{(i + 1)!} \int_M u (\chi + \tilde \chi + \sqrt{-1 }\partial\bar\partial u)^j \wedge (\chi + \tilde\chi)^{i - j} \wedge \omega^{n - i} \\
	&= \frac{1}{i + 1} \sum^i_{j = 0}  \int_M u (\chi + \tilde \chi + \sqrt{-1 }\partial\bar\partial u)^j \wedge (\chi + \tilde\chi)^{i - j} \wedge \omega^{n - i}
	.
\end{aligned}
\end{equation}
From \eqref{equality-3-11}, 
\begin{equation}
\label{equality-3-12}
\begin{aligned}
	&\quad 
	c J_n (\chi + \tilde \chi, u) - n J_m (\chi + \tilde \chi, u) \\
	&=
	\frac{c}{n + 1} \sum^n_{j=0} \int_M u (\chi + \tilde \chi + \sqrt{-1 }\partial\bar\partial u)^j \wedge (\chi + \tilde\chi)^{n - j} \\
	&\qquad - \frac{n}{m + 1} \sum^m_{j = 0} \int_M u (\chi + \tilde \chi + \sqrt{-1 }\partial\bar\partial u)^j \wedge (\chi + \tilde\chi)^{m - j} \wedge \omega^{n - m} 
	.
\end{aligned}
\end{equation}
In particular,
\begin{equation}
\label{equality-2-9}
\begin{aligned}
	&\quad 
	c J_n (\chi + \tilde \chi, u + C) - n J_m (\chi + \tilde \chi, u + C) \\
%	&=
%	\frac{c}{n + 1} \sum^n_{j=0} \int_M (u + C) (\chi + \tilde \chi + \sqrt{-1 }\partial\bar\partial u)^j \wedge (\chi + \tilde\chi)^{n - j} \\
%	&\qquad - \frac{n}{m + 1} \sum^m_{j = 0} \int_M (u + C) (\chi + \tilde \chi + \sqrt{-1 }\partial\bar\partial u)^j \wedge (\chi + \tilde\chi)^{m - j} \wedge \omega^{n - m} 
%	\\
%	&= 
%	\frac{c}{n + 1} \sum^n_{j=0} \int_M u  (\chi + \tilde \chi + \sqrt{-1 }\partial\bar\partial u)^j \wedge (\chi + \tilde\chi)^{n - j} \\
%	&\qquad - \frac{n}{m + 1} \sum^m_{j = 0} \int_M u  (\chi + \tilde \chi + \sqrt{-1 }\partial\bar\partial u)^j \wedge (\chi + \tilde\chi)^{m - j} \wedge \omega^{n - m} 
%	\\
%	&\qquad + 	\frac{c C}{n + 1} \sum^n_{j=0} \int_M   (\chi + \tilde \chi + \sqrt{-1 }\partial\bar\partial u)^j \wedge (\chi + \tilde\chi)^{n - j} \\
%		&\qquad - \frac{n C}{m + 1} \sum^m_{j = 0} \int_M   (\chi + \tilde \chi + \sqrt{-1 }\partial\bar\partial u)^j \wedge (\chi + \tilde\chi)^{m - j} \wedge \omega^{n - m} \\
%	&= 
%	c J_n (\chi + \tilde \chi, u) - n J_m (\chi + \tilde \chi, u)	\\
%	&\qquad + 	\frac{c C}{n + 1} \sum^n_{j=0} \int_M    (\chi + \tilde\chi)^{n}  - \frac{n C}{m + 1} \sum^m_{j = 0} \int_M    (\chi + \tilde\chi)^{m } \wedge \omega^{n - m} \\
	&= 
	c J_n (\chi + \tilde \chi, u) - n J_m (\chi + \tilde \chi, u)
	\\
	&\qquad +  C \left(c  \int_M    (\chi + \tilde\chi)^{n}  -  n   \int_M    (\chi + \tilde\chi)^{m } \wedge \omega^{n - m} \right) \\
	&=
	c J_n (\chi + \tilde \chi, u) - n J_m (\chi + \tilde \chi, u) 
	.
\end{aligned}
\end{equation}
If $\varphi$ is uniformly bounded which will be proven in Section~\ref{C0}, we derive from \eqref{equality-3-12} that
\begin{equation*}
\label{inequality-2-9}
	c J_n (\chi + \tilde \chi, \varphi) - n J_m (\chi + \tilde \chi, \varphi) \leq C ,
\end{equation*}
where $C$ depends on the uniform bound of $\varphi$. 
Along the solution flow $\varphi (t)$ to  Equation~\eqref{equation-3-1},
\begin{equation*}
\label{inequality-3-14}
\begin{aligned}
	\frac{d}{dt} \left(c J_n (\chi + \tilde \chi, \varphi) - n J_m (\chi + \tilde \chi, \varphi)\right)
	&=
	c \int_M \frac{\partial \varphi}{\partial t} (\chi + \tilde \chi + \sqrt{-1} \partial\bar\partial \varphi)^n \\
	&\quad  - n \int_M \frac{\partial \varphi}{\partial t} (\chi + \tilde \chi + \sqrt{-1} \partial\bar\partial \varphi)^m \wedge \omega^{n - m} \\
	&= \int_M \left( \frac{\partial \varphi}{\partial t} \right)^2 (\chi + \tilde \chi + \sqrt{-1} \partial\bar\partial\varphi)^n \\
	&\geq 0
	.
\end{aligned}
\end{equation*}
Therefore, for any $t > 0$,
%\begin{equation}
%\int^{+\infty}_0	\left(\int_M \left(\frac{\partial \varphi}{\partial t}\right)^2 (\chi + \tilde \chi + \sqrt{-1} \partial\bar\partial \varphi)^n\right) dt \leq C .
%\end{equation}
\begin{equation*}
C
\geq
c J_n (\chi + \tilde \chi, \varphi (t)) - n J_m (\chi + \tilde \chi, \varphi (t)) 
\geq
c J_n (\chi + \tilde \chi, \varphi_0) - n J_m (\chi + \tilde \chi, \varphi_0)
.
\end{equation*}
In particular,
\begin{equation}
\label{inequality-2-10}
\begin{aligned}
&
	\int^{+\infty}_0 \left(\int_M \left(\frac{\partial \varphi}{\partial t}\right)^2 (\chi + \tilde \chi + \sqrt{-1}\partial\bar\partial\varphi)^n \right) dt \\
%	&= 
%	\lim_{t\to +\infty} \left(c J_n (\chi + \tilde \chi, \varphi (t)) - n J_m (\chi + \tilde \chi, \varphi (t)) \right) - \left(c J_n (\chi + \tilde \chi, \varphi_0) - n J_m (\chi + \tilde \chi, \varphi_0)\right) 
	&\leq 
	C - \left(c J_n (\chi + \tilde \chi, \varphi_0) - n J_m (\chi + \tilde \chi, \varphi_0)\right)
	< + \infty
	. 
\end{aligned}
\end{equation}
Similarly, we also have
%along the solution flow $\varphi(t)$ to Equation~\eqref{equation-3-1},
\begin{equation}
\label{equality-2-16}
\begin{aligned}
	\frac{d}{dt} J_n (\chi + \tilde \chi, \varphi)
	&=
	\int_M \frac{\partial \varphi}{\partial t} (\chi + \tilde\chi + \sqrt{-1} \partial\bar\partial \varphi)^n \\
%	&= 
%	\int_M \left(c - \frac{n (\chi + \tilde \chi + \sqrt{-1} \partial\bar\partial\varphi)^m \wedge \omega^{n - m}}{(\chi + \tilde \chi +\ sqrt{-1} \partial\bar\partial\varphi)^n}\right) (\chi + \tilde\chi + \sqrt{-1} \partial\bar\partial \varphi)^n \\
%	&= c \int_M (\chi + \tilde\chi + \sqrt{-1} \partial\bar\partial \varphi)^n  -  n \int_M (\chi + \tilde \chi + \sqrt{-1} \partial\bar\partial \varphi)^m \wedge \omega^{n - m} \\
	&= c \int_M (\chi + \tilde\chi )^n  -  n \int_M (\chi + \tilde \chi  )^m \wedge \omega^{n - m} \\
	&= 0
	.
\end{aligned}
\end{equation}
Identity~\eqref{equality-2-16} implies that
\begin{equation}
\label{equality-2-17}
	J_n (\chi + \tilde \chi, \varphi) = J_n (\chi + \tilde \chi, \varphi_0) ,
\end{equation}
and we then have that for any $t \in [0,T)$, 
\begin{equation}
\label{inequality-2-18}
	\sup_M \frac{\partial\varphi}{\partial t} \geq 0 \geq \inf_M \frac{\partial \varphi}{\partial t} .
\end{equation}
%as an immediate consequence.

%\newpage

\medskip
\section{$C^0$ estimate}
\label{C0}

In this section, we shall prove $C^0$ estimate of admissible solution to Equation~\eqref{equation-3-1}. We shall adpat the first method in \cite{FangLaiSongWeinkove2014}.
\begin{theorem}
Let $\varphi (z,t)$ be the admissible solution to Equation~\eqref{equation-3-1}.
Then there is a constant $C > 0$ independent of time $t$ such that
\begin{equation*}
	\Vert \varphi (z,t) \Vert_{L^\infty} \leq C .
\end{equation*}
\end{theorem}

\begin{proof}

According to \cite{Sun202210}, there is a pluripotential solution $\psi$ to Equation~\eqref{elliptic-equation} with $ess\sup_M \psi = 0$. 
%\begin{equation*}
%	c (\chi + \tilde \chi + \sqrt{- 1} \partial\bar\partial \psi)^n = n (\chi + \tilde \chi + \sqrt{-1} \partial\bar\partial \psi)^m \wedge \omega^{n - m} , \qquad ess \sup \psi = 0 .
%\end{equation*}
The solution $\psi$ is uniformly bounded and smooth in $Amp (\tilde \chi)$. 
There is a function $\rho$ which is smooth in $Amp (\tilde\chi)$ with analytic singularities 
and $\tilde\chi + \sqrt{-1}\partial\bar\partial\rho \geq \delta \chi $ for some $\delta > 0$. 
The existence of $\rho$ is guaranteed by the work of Boucksom~\cite{Boucksom2004}.

We shall consider the following function: for $\epsilon \in \left(0,\frac{1}{2}\right)$,
\begin{equation*}
	\theta_{\epsilon} := \varphi - \psi - \epsilon (\psi - \rho + cm t) .
\end{equation*}
Function $\theta_\epsilon$ is smooth in $Amp (\tilde \chi)$, and approaches $-\infty$ along $M \setminus Amp (\tilde \chi)$.  
For a fixed time period of $T > 0$,  $\theta_\epsilon$ must reach its maximal value at some point $(z_{max},t_{max}) \in Amp(\tilde \chi) \times [0,T]$. If point $(z_{max},t_{max}) \in Amp(\tilde \chi) \times (0,T]$,
\begin{equation}
\label{inequality-3-5}
	\frac{\partial \theta_\epsilon}{\partial t} \geq 0 , \qquad \text{ and } \qquad  \sqrt{-1} \partial\bar\partial \theta_\epsilon \leq 0 .
\end{equation}
From the second inequality in \eqref{inequality-3-5}, we reduce that at $(z_{max},t_{max})$
\begin{equation*}
\begin{aligned}
	&\quad \chi + \tilde \chi + \sqrt{-1} \partial\bar\partial \varphi \\
%	&= (1 + \epsilon) (\chi + \tilde \chi + \sqrt{-1} \partial\bar\partial \psi) - \epsilon (\chi + \tilde \chi) + \sqrt{- 1} \partial\bar\partial \varphi - (1 + \epsilon) \sqrt{-1} \partial\bar\partial \psi \\
%	&= (1 + \epsilon) (\chi + \tilde \chi + \sqrt{-1} \partial\bar\partial \psi) - \epsilon (\chi + \tilde \chi + \sqrt{-1}\partial\bar\partial \rho) \\
%	&\qquad + \epsilon \sqrt{-1} \partial\bar\partial \rho +  \sqrt{- 1} \partial\bar\partial \varphi - (1 + \epsilon) \sqrt{-1} \partial\bar\partial \psi \\
%	&= (1 + \epsilon) (\chi + \tilde \chi + \sqrt{-1} \partial\bar\partial \psi) - \epsilon (\chi + \tilde \chi + \sqrt{-1}\partial\bar\partial \rho) \\
%	&\qquad + \epsilon \sqrt{-1} \partial\bar\partial \rho +  \sqrt{- 1} \partial\bar\partial \varphi - (1 + \epsilon) \sqrt{-1} \partial\bar\partial \psi  - \sqrt{-1} \partial\bar\partial (cm t) \\
%	&= (1 + \epsilon) (\chi + \tilde \chi + \sqrt{-1} \partial\bar\partial \psi) - \epsilon (\chi + \tilde \chi + \sqrt{-1}\partial\bar\partial \rho) \\
%	&\qquad  + \sqrt{-1} \partial\bar\partial \left( \varphi - \psi - \epsilon (\psi - \rho + cm t)\right)  \\
	&= (1 + \epsilon) (\chi + \tilde \chi + \sqrt{-1} \partial\bar\partial \psi) - \epsilon (\chi + \tilde \chi + \sqrt{-1}\partial\bar\partial \rho) + \sqrt{-1} \partial\bar\partial \theta_\epsilon  \\
	&< (1 + \epsilon) (\chi + \tilde \chi + \sqrt{-1} \partial\bar\partial \psi) ,
\end{aligned}
\end{equation*}
and hence 
\begin{equation}
\label{inequality-3-7}
\begin{aligned}
	\quad \frac{(\chi + \tilde \chi + \sqrt{-1} \partial\bar\partial \varphi)^n}{ n (\chi + \tilde \chi + \sqrt{-1} \partial\bar\partial \varphi)^m \wedge \omega^{n - m}} 
	&<
	\frac{ (1 + \epsilon)^n (\chi + \tilde \chi + \sqrt{-1} \partial\bar\partial\psi)^n}{(1 + \epsilon)^m n (\chi + \tilde \chi + \sqrt{-1} \partial\bar\partial \psi)^m \wedge \omega^{n - m}} \\
	&= 
	 \frac{(1 + \epsilon)^{n - m}}{c} .
\end{aligned}
\end{equation}
Substituting the first inequality in \eqref{inequality-3-5} and \eqref{inequality-3-7} into Equation~\eqref{equation-3-1}, 
\begin{equation*}
\begin{aligned}
	\frac{\partial \theta_\epsilon}{\partial t}
	&= \frac{\partial \varphi}{\partial t} - \epsilon cm \\
	&= c - \frac{n (\chi + \tilde \chi + \sqrt{-1} \partial\bar\partial\varphi)^{m } \wedge \omega^{n - m}}{(\chi + \tilde \chi + \sqrt{-1} \partial\bar\partial\varphi)^n} - \epsilon cm \\
	&< c - \frac{c}{(1 + \epsilon)^m} - \epsilon cm \\
%	&= c   \frac{m}{(1 + \xi)^{m + 1}} \epsilon - \epsilon cm  \\
%	&< cm \epsilon - \epsilon cm \\
	&< 0 ,
\end{aligned}
\end{equation*}
which is a contradiction!
%\begin{equation}
%\begin{aligned}
%	\frac{\partial \theta_\epsilon}{\partial t}
%	&= \frac{\partial \varphi}{\partial t} - \epsilon A \\
%	&= c - \frac{n (\chi + \tilde \chi + \sqrt{-1} \partial\bar\partial\varphi)^{m } \wedge \omega^{n - m}}{(\chi + \tilde \chi + \sqrt{-1} \partial\bar\partial\varphi)^n} - \epsilon A \\
%	&=
%	c - \frac{n \left( (1 + \epsilon)(\chi + \tilde \chi + \sqrt{-1}\partial\bar\partial\psi ) - \epsilon (\chi + \tilde \chi + \sqrt{-1} \partial\bar\partial \rho)    + \sqrt{-1} \partial\bar\partial \theta_\epsilon\right)^m \wedge \omega^{n - m}}{\left( (1 + \epsilon)(\chi + \tilde \chi + \sqrt{-1}\partial\bar\partial\psi ) - \epsilon (\chi + \tilde \chi + \sqrt{-1} \partial\bar\partial \rho)    + \sqrt{-1} \partial\bar\partial \theta_\epsilon\right)^n} - \epsilon A \\
%	&\leq 
%	c - \frac{1}{(1 + \epsilon)^m} \frac{n(\chi + \tilde \chi + \sqrt{-1} \partial\bar\partial\psi)^m \wedge \omega^{n - m} }{(\chi + \tilde \chi + \sqrt{-1} \partial\bar\partial \psi)^n} -\epsilon A \\
%	&= c - \frac{c}{ (1 + \epsilon)^m } - \epsilon A \\
%	&< 0 .
%\end{aligned}
%\end{equation}
Now we can conclude that $z_{max} \in Amp(\tilde \chi)$ and $t_{max} = 0$. 
As a result,
%\begin{equation}
%	\theta_\epsilon 
%	\leq 
%	\sup_M \theta_{\epsilon} (z, 0)
%	=
%	\sup_M \left( \varphi(z,0) - (1 + \epsilon) \psi + \epsilon \rho \right),
%\end{equation}
%and hence
\begin{equation}
\label{inequality-3-9}
\begin{aligned}
	\varphi (z,t)
	&\leq 
	 \varphi(z_{max},0) - (1 + \epsilon) \psi (z_{max}) + \epsilon \rho (z_{max}) + (1 + \epsilon) \psi (z) - \epsilon \rho  (z) + \epsilon cm t \\
	 &\leq 
	 \varphi(z_{max},0) - (1 + \epsilon) \psi (z_{max}) - \epsilon \rho  (z) + \epsilon cm t 
	 .
\end{aligned}
\end{equation}
Letting $\epsilon \to 0 +$ in \eqref{inequality-3-9},
\begin{equation*}
	\varphi (z,t) \leq \varphi (z_{max},0) - \psi (z_{max}) \leq \max_M \varphi_0 + \Vert \psi \Vert_{L^\infty}.
\end{equation*}
Similarly, we also obtain
\begin{equation*}
	\varphi (z,t) \geq \varphi (z_{min},0) - \psi (z_{min}) \geq \min_M \varphi_0 .
\end{equation*}

\end{proof}

%\newpage

\medskip

\section{Second order estimate}

In this section, we shall prove partial second order estimates of admissible solutions to Equation~\eqref{equation-3-1} following the argument in \cite{FangLaiMa2011}\cite{Sun2015p}, which include uniform estimate in $Amp(\tilde \chi)$ and global estimate dependent on $t$. To well understand  the dependence of constants and coefficients in two cases, we provide the proofs in details here.

\medskip

\subsection{Uniform second order estimate}
\label{uniform-C2}

We shall consider the function
\begin{equation}
	\ln w - A (\varphi - \rho), \qquad \text{ where } w := S_1 (\chi + \tilde\chi + \sqrt{-1} \partial\bar\partial \varphi).
\end{equation}
Since $\rho$ approaches $-\infty$ along $M \setminus Amp (\tilde \chi)$, function $\ln w - A (\varphi - \rho)$ can reach its maximal value on $M \times [0,t']$, for a fixed time $0 < t' < T$. 
Suppose that $\ln w + A (\varphi - \rho)$ achieves its maximum at point $(z_{max},t_{max}) \in Amp(\tilde \chi) \times (0,t']$. We choose a local chart near $z_{max}$ such that $\omega_{i\bar j} = \delta_{ij}$ and $X_{i\bar j}$ is diagonal at $z_{max}$ when $t = t_{max}$. 
Therefore, at $(z_{max},t_{max})$
\begin{equation}
\label{inequality-3-19}
	\frac{\partial_t w}{w} - A \partial_t \varphi \geq 0,
\end{equation}
\begin{equation}
\label{equation-3-20}
	\frac{\partial_l w}{w} - A (\varphi_l - \rho_l) = 0,
\end{equation}
and
\begin{equation}
\label{inequality-3-21}
	\frac{\bar\partial_l \partial_l w}{w} - \frac{\bar\partial_l w \partial_l w}{w^2} - A (\varphi_{l\bar l} - \rho_{l\bar l}) \leq 0 .
\end{equation}
From \eqref{inequality-3-21}, we obtain
\begin{equation}
\label{inequality-3-22}
\begin{aligned}
	0	
	&\geq
%	\frac{\bar\partial_l \partial_l w}{w} - \frac{\bar\partial_l w \partial_l w}{w^2} - A (\varphi_{l\bar l} - \rho_{l\bar l}) \\
%	&= 
%	\frac{1}{w} \sum_i X_{i\bar il\bar l}  - \frac{\bar\partial_l w \partial_l w}{w^2} - A (\varphi_{l\bar l} - \rho_{l\bar l}) \\
%	&= 
	\frac{1}{w} \sum_i (X_{l\bar li\bar i} - R_{i\bar il\bar l} X_{l\bar l} + R_{l\bar li\bar i} X_{i\bar i} + G_{l\bar li\bar i}) - \frac{\bar\partial_l w \partial_l w}{w^2} - A(\varphi_{l\bar l} - \rho_{l\bar l}) ,
\end{aligned}
\end{equation}
where
\begin{equation*}
%\begin{aligned}
G_{l\bar li\bar i} 
=
(\chi_{i\bar il\bar l} + \tilde \chi_{i\bar il\bar l}) - (\chi_{l\bar li\bar i}  + \tilde\chi_{l\bar li\bar i} )+ \sum_k R_{i\bar il\bar k} (\chi_{k\bar l}  + \tilde \chi_{k\bar l} )- \sum_k R_{l\bar li\bar k} (\chi_{k\bar i}  + \tilde \chi_{k\bar i} ) .
%\end{aligned}
\end{equation*}
Multiplying \eqref{inequality-3-22} by $S_{n - m - 1;l} (X^{-1}) (X^{l\bar l})^2$ and summing them up along with index $l$,
\begin{equation}
\label{inequality-3-24}
\begin{aligned}
	0 
	&\geq
%	\frac{1}{w} \sum_l S_{n - m - 1;l} (X^{-1}) (X^{l\bar l})^2 \bar\partial_l \partial_l w - \frac{1}{w^2} \sum_l S_{n - m - 1;l} (X^{-1}) (X^{l\bar l})^2 \partial_l w\bar\partial_l w \\
%	&\qquad - A \sum_l S_{n - m - 1;l} (X^{-1}) (X^{l\bar l})^2 (\varphi_{l\bar l} - \rho_{l\bar l}) \\
%	&=
%	\frac{1}{w} \sum_{i,l} S_{n - m - 1;l} (X^{-1}) (X^{l\bar l})^2 (X_{l\bar li\bar i} - R_{i\bar il\bar l} X_{l\bar l} + R_{l\bar li\bar i} X_{i\bar i} + G_{l\bar li\bar i})  \\
%	&\qquad - \frac{1}{w^2} \sum_l S_{n - m - 1;l} (X^{-1}) (X^{l\bar l})^2 \partial_l w\bar\partial_l w \\
%	&\qquad - A \sum_l S_{n - m - 1;l} (X^{-1}) (X^{l\bar l})^2 (\varphi_{l\bar l} - \rho_{l\bar l}) \\
%	&=
	\frac{1}{w} \sum_{i,l} S_{n - m - 1;l} (X^{-1}) (X^{l\bar l})^2 X_{l\bar li\bar i} - \frac{1}{w} \sum_{i,l} S_{n - m - 1;l} (X^{-1}) X^{l\bar l} R_{i\bar il\bar l} \\
	&\qquad + \frac{1}{w} \sum_{i,l} S_{n - m - 1;l} (X^{-1}) (X^{l\bar l})^2 R_{l\bar li\bar i} X_{i\bar i} + \frac{1}{w} \sum_{i,l} S_{n - m - 1;l} (X^{-1}) (X^{l\bar l})^2 G_{l\bar li\bar i} \\
	&\qquad - \frac{1}{w^2} \sum_l S_{n - m - 1;l} (X^{-1}) (X^{l\bar l})^2 \partial_l w\bar\partial_l w \\
	&\qquad - A \sum_l S_{n - m - 1;l} (X^{-1}) (X^{l\bar l})^2 (\varphi_{l\bar l} - \rho_{l\bar l}) 
	.
\end{aligned}
\end{equation}
Substituting \eqref{inequality-3-15} and \eqref{inequality-3-19} into \eqref{inequality-3-24},
\begin{equation}
\label{inequality-3-25}
\begin{aligned}
	0 &\geq
		 \frac{1}{w}\sum_{i,j,l} S_{n - m -1;i} (X^{-1}) (X^{i\bar i})^2 X^{j\bar j} X_{j\bar i\bar l} X_{i\bar jl} + \frac{C^m_n}{n w} \sum_l \partial_t \varphi_{l\bar l}  \\
		&\qquad - \frac{1}{w} \sum_{i,l} S_{n - m - 1;l} (X^{-1}) X^{l\bar l} R_{i\bar il\bar l}  + \frac{1}{w} \sum_{i,l} S_{n - m - 1;l} (X^{-1}) (X^{l\bar l})^2 R_{l\bar li\bar i} X_{i\bar i}  \\
		&\qquad + \frac{1}{w} \sum_{i,l} S_{n - m - 1;l} (X^{-1}) (X^{l\bar l})^2 G_{l\bar li\bar i}  - \frac{1}{w^2} \sum_l S_{n - m - 1;l} (X^{-1}) (X^{l\bar l})^2 \partial_l w\bar\partial_l w \\
		&\qquad - A \sum_l S_{n - m - 1;l} (X^{-1}) (X^{l\bar l})^2 (\varphi_{l\bar l} - \rho_{l\bar l}) \\
%		&=
%		\frac{1}{w}\sum_{i,j,l} S_{n - m -1;i} (X^{-1}) (X^{i\bar i})^2 X^{j\bar j} X_{j\bar i\bar l} X_{i\bar jl} +  \frac{C^m_n}{n }  \frac{\partial_t w}{w}  \\
%		&\qquad - \frac{1}{w} \sum_{i,l} S_{n - m - 1;l} (X^{-1}) X^{l\bar l} R_{i\bar il\bar l}  + \frac{1}{w} \sum_{i,l} S_{n - m - 1;l} (X^{-1}) (X^{l\bar l})^2 R_{l\bar li\bar i} X_{i\bar i}  \\
%		&\qquad + \frac{1}{w} \sum_{i,l} S_{n - m - 1;l} (X^{-1}) (X^{l\bar l})^2 G_{l\bar li\bar i}  - \frac{1}{w^2} \sum_l S_{n - m - 1;l} (X^{-1}) (X^{l\bar l})^2 \partial_l w\bar\partial_l w \\
%		&\qquad - A \sum_l S_{n - m - 1;l} (X^{-1}) (X^{l\bar l})^2 (\varphi_{l\bar l} - \rho_{l\bar l}) \\
		&\geq
		\frac{1}{w}\sum_{i,j,l} S_{n - m -1;i} (X^{-1}) (X^{i\bar i})^2 X^{j\bar j} X_{j\bar i\bar l} X_{i\bar jl}  +  A \frac{C^m_n}{n }  \partial_t \varphi  \\
		&\qquad - \frac{1}{w} \sum_{i,l} S_{n - m - 1;l} (X^{-1}) X^{l\bar l} R_{i\bar il\bar l}  + \frac{1}{w} \sum_{i,l} S_{n - m - 1;l} (X^{-1}) (X^{l\bar l})^2 R_{l\bar li\bar i} X_{i\bar i}  \\
		&\qquad + \frac{1}{w} \sum_{i,l} S_{n - m - 1;l} (X^{-1}) (X^{l\bar l})^2 G_{l\bar li\bar i}  - \frac{1}{w^2} \sum_l S_{n - m - 1;l} (X^{-1}) (X^{l\bar l})^2 \partial_l w\bar\partial_l w \\
		&\qquad - A \sum_l S_{n - m - 1;l} (X^{-1}) (X^{l\bar l})^2 (\varphi_{l\bar l} - \rho_{l\bar l}) 
%		\\
%		&=
%		\frac{1}{w}\sum_{i,j,l} S_{n - m -1;i} (X^{-1}) (X^{i\bar i})^2 X^{j\bar j} X_{j\bar i\bar l} X_{i\bar jl}   + A  \left(\frac{C^m_n c}{n} -   S_{n - m} (X^{-1})\right) \\
%		&\qquad - \frac{1}{w} \sum_{i,l} S_{n - m - 1;l} (X^{-1}) X^{l\bar l} R_{i\bar il\bar l}  + \frac{1}{w} \sum_{i,l} S_{n - m - 1;l} (X^{-1}) (X^{l\bar l})^2 R_{l\bar li\bar i} X_{i\bar i}  \\
%		&\qquad + \frac{1}{w} \sum_{i,l} S_{n - m - 1;l} (X^{-1}) (X^{l\bar l})^2 G_{l\bar li\bar i}  - \frac{1}{w^2} \sum_l S_{n - m - 1;l} (X^{-1}) (X^{l\bar l})^2 \partial_l w\bar\partial_l w \\
%		&\qquad - A \sum_l S_{n - m - 1;l} (X^{-1}) (X^{l\bar l})^2 (\varphi_{l\bar l} - \rho_{l\bar l}) 
		.		
\end{aligned}
\end{equation}
Since
\begin{equation*}
	X_{i\bar jl} X_{j\bar i\bar l} - 2 \mathfrak{Re} \left\{X_{i\bar jl} X_{l\bar j} \frac{\bar\partial_i w}{w}\right\} + X_{j\bar l} X_{l\bar j} \frac{\partial_i w \bar\partial_i w}{w^2} \geq 0,
\end{equation*}
we obtain that
\begin{equation}
\label{inequality-3-27}
\begin{aligned}
	0
	&\leq
%	\frac{1}{w} \sum_{i,j,l} S_{n - m -1;i} (X^{-1}) (X^{i\bar i})^2 X^{j\bar j} X_{j\bar i\bar l} X_{i\bar jl} \\
%	&\qquad - \frac{2}{w} \sum_{i,j,l} S_{n - m -1;i} (X^{-1}) (X^{i\bar i})^2 X^{j\bar j} \mathfrak{Re} \left\{X_{i\bar jl} X_{l\bar j} \frac{\bar\partial_i w}{w}\right\} \\
%	&\qquad + \frac{1}{w} \sum_{i,j,;} S_{n - m -1;i} (X^{-1}) (X^{i\bar i})^2 X^{j\bar j} X_{j\bar l} X_{l\bar j} \frac{\partial_i w\bar\partial_i w}{w^2} \\
%	&=
%	\frac{1}{w} \sum_{i,j,l} S_{n - m -1;i} (X^{-1}) (X^{i\bar i})^2 X^{j\bar j} X_{j\bar i\bar l} X_{i\bar jl} \\	
%	&\qquad - \frac{2}{w} \sum_{i,j} S_{n - m -1;i} (X^{-1}) (X^{i\bar i})^2   \mathfrak{Re} \left\{X_{i\bar jj}   \frac{\bar\partial_i w}{w}\right\} \\
%	&\qquad + \frac{1}{w^2} \sum_{i} S_{n - m -1;i} (X^{-1}) (X^{i\bar i})^2  \partial_i w\bar\partial_i w  \\
%	&=
%	\frac{1}{w} \sum_{i,j,l} S_{n - m -1;i} (X^{-1}) (X^{i\bar i})^2 X^{j\bar j} X_{j\bar i\bar l} X_{i\bar jl} \\	
%	&\qquad - \frac{2}{w^2} \sum_{i,j} S_{n - m -1;i} (X^{-1}) (X^{i\bar i})^2   \mathfrak{Re} \left\{X_{j\bar ji}    \bar\partial_i w \right\} \\
%	&\qquad + \frac{1}{w^2} \sum_{i} S_{n - m -1;i} (X^{-1}) (X^{i\bar i})^2  \partial_i w\bar\partial_i w  \\
%	&=
%	\frac{1}{w} \sum_{i,j,l} S_{n - m -1;i} (X^{-1}) (X^{i\bar i})^2 X^{j\bar j} X_{j\bar i\bar l} X_{i\bar jl} \\	
%	&\qquad - \frac{2}{w^2} \sum_{i} S_{n - m -1;i} (X^{-1}) (X^{i\bar i})^2    \partial_i w   \bar\partial_i w   \\
%	&\qquad + \frac{1}{w^2} \sum_{i} S_{n - m -1;i} (X^{-1}) (X^{i\bar i})^2  \partial_i w\bar\partial_i w  \\
%	&=
	\frac{1}{w} \sum_{i,j,l} S_{n - m -1;i} (X^{-1}) (X^{i\bar i})^2 X^{j\bar j} X_{j\bar i\bar l} X_{i\bar jl} \\	
	&\qquad - \frac{1}{w^2} \sum_{i} S_{n - m -1;i} (X^{-1}) (X^{i\bar i})^2    \partial_i w   \bar\partial_i w   
	.
\end{aligned}
\end{equation}
Substituting \eqref{inequality-3-27} into \eqref{inequality-3-25},
\begin{equation}
\label{inequality-3-28}
\begin{aligned}
	0
	&\geq 
	  A  \left(\frac{C^m_n c}{n} -   S_{n - m} (X^{-1})\right) + \frac{1}{w} \sum_{i,l} S_{n - m - 1;l} (X^{-1}) (X^{l\bar l})^2 G_{l\bar li\bar i} \\
		&\qquad - \frac{1}{w} \sum_{i,l} S_{n - m - 1;l} (X^{-1}) X^{l\bar l} R_{i\bar il\bar l}  + \frac{1}{w} \sum_{i,l} S_{n - m - 1;l} (X^{-1}) (X^{l\bar l})^2 R_{l\bar li\bar i} X_{i\bar i}  \\
		&\qquad - A \sum_l S_{n - m - 1;l} (X^{-1}) (X^{l\bar l})^2 (\varphi_{l\bar l} - \rho_{l\bar l}) 
		.
\end{aligned}
\end{equation}

Now we need to estimate the positive term in \eqref{inequality-3-28}, following the argument of \cite{FangLaiMa2011}.
\begin{lemma}
\label{lemma-3-1}
Let $\varphi \in C^2(M)$ satisfy $\chi + \tilde \chi + \sqrt{-1} \partial\bar\partial \varphi > 0$. There are positive constants $N$ and $\theta$ such that when $w > N$ at a point $z \in Amp (\tilde \chi)$,
\begin{equation}
\begin{aligned}
	&\quad \sum_i S_{n - m - 1;i} (X^{-1}) (X^{i\bar i})^2 (\varphi_{i\bar i} - \rho_{i\bar i}) \\
	&\leq \frac{C^m_n c}{n} -  S_{n - m} (X^{-1}) - \theta - \theta \sum_i S_{n - m - 1;i} (X^{-1}) (X^{i\bar i})^2 
%	\\
%	&= \frac{C^m_n c}{n} \frac{\partial\varphi}{\partial t} - \theta - \theta \sum_i S_{n - m - 1;i} (X^{-1}) (X^{i\bar i})^2 
	,
\end{aligned}
\end{equation}
under coordinates around $z$ such that $\omega_{i\bar  j} = \delta_{ij}$ and $X_{i\bar j}$ is diagonal at $z$.

\end{lemma}
\begin{proof}

Without loss of generality, we may assume that $\chi \geq \kappa \omega$ for $\kappa > 0$, and $X_{1\bar 1} \geq \cdots \geq X_{n\bar n}$. We calculate that
\begin{equation}
\label{inequality-3-30}
\begin{aligned}
	&\quad \sum_i S_{n - m - 1;i} (X^{-1}) (X^{i\bar i})^2 (\varphi_{i\bar i} - \rho_{i\bar i}) \\
%	&=  \sum_i S_{n - m - 1;i} (X^{-1}) X^{i\bar i} - \sum_i S_{n - m - 1;i} (X^{-1}) (X^{i\bar i})^2 (\chi_{i\bar i} + \tilde \chi_{i\bar i} + \rho_{i\bar i}) \\
	&\leq \sum_i S_{n - m - 1;i} (X^{-1}) X^{i\bar i} - (1 + \delta) \sum_i S_{n - m - 1;i} (X^{-1}) (X^{i\bar i})^2 \chi_{i\bar i}  \\
%	&= (n - m) S_{n - m} (X^{-1}) - (1 + \delta) \sum_i S_{n - m - 1;i} (X^{-1}) (X^{i\bar i})^2 \chi_{i\bar i} \\
	&\leq (n - m) S_{n - m} (X^{-1}) - (1 + \delta) \kappa \sum_i S_{n - m - 1;i} (X^{-1}) (X^{i\bar i})^2 \\
%	&= (n - m) S_{n - m} (X^{-1}) - \frac{(1 + \delta) \kappa}{2} \sum_i S_{n - m - 1;i} (X^{-1}) (X^{i\bar i})^2 \\
%	&\qquad - \frac{(1 + \delta) \kappa}{2} \sum_i S_{n - m - 1;i} (X^{-1}) (X^{i\bar i})^2 \\
%	&= (n - m) S_{n - m} (X^{-1}) - \frac{(1 + \delta) \kappa}{2} \sum_i \left( S_{n - m } (X^{-1}) - S_{n - m;i } (X^{-1}) \right) X^{i\bar i} \\
%	&\qquad - \frac{(1 + \delta) \kappa}{2} \sum_i S_{n - m - 1;i} (X^{-1}) (X^{i\bar i})^2 \\
%	&= (n - m) S_{n - m} (X^{-1}) - \frac{(1 + \delta) \kappa}{2} \left( S_{n - m } (X^{-1}) S_1 (X^{-1}) - (n - m + 1) S_{n - m + 1} (X^{-1}) \right) \\
%	&\qquad - \frac{(1 + \delta) \kappa}{2} \sum_i S_{n - m - 1;i} (X^{-1}) (X^{i\bar i})^2 \\
	&\leq 
%	(n - m) S_{n - m} (X^{-1}) - \frac{(1 + \delta) \kappa}{2} \frac{n - m}{n} S_{n - m} (X^{-1}) S_1 (X^{-1}) \\
%	&\qquad - \frac{(1 + \delta) \kappa}{2} \sum_i S_{n - m - 1;i} (X^{-1}) (X^{i\bar i})^2 \\
%	&=
	(n - m) \left(1 - \frac{(1 + \delta) \kappa}{2 n} S_1 (X^{-1})\right)S_{n - m} (X^{-1}) \\
	&\qquad - \frac{(1 + \delta) \kappa}{2} \sum_i S_{n - m - 1;i} (X^{-1}) (X^{i\bar i})^2
	.
\end{aligned}
\end{equation}
If $S_1 (X^{-1}) \geq \frac{2n (n - m + 1)}{(n - m) (1 + \delta) \kappa}$, it is from \eqref{inequality-3-30} that
\begin{equation}
\begin{aligned}
	&\quad \sum_i S_{n - m - 1;i} (X^{-1}) (X^{i\bar i})^2 (\varphi_{i\bar i} - \rho_{i\bar i}) \\
	&\leq - S_{n - m} (X^{-1}) - \frac{(1 + \delta) \kappa}{2} \sum_i S_{n - m - 1;i} (X^{-1}) (X^{i\bar i})^2
	.
\end{aligned}
\end{equation}
If $S_1 (X^{-1}) < \frac{2n (n - m + 1)}{(n - m) (1 + \delta) \kappa}$, then
\begin{equation}
	X_{1\bar 1} \geq \cdots \geq X_{n\bar n} > \frac{1}{S_1 (X^{-1})} > \frac{(n - m) (1 + \delta) \kappa}{2n (n - m + 1)} .
\end{equation}
So we have
\begin{equation}
\begin{aligned}
	&\quad \sum_i S_{n - m - 1;i} (X^{-1}) (X^{i\bar i})^2 (\varphi_{i\bar i} - \rho_{i\bar i}) \\
%	&= \sum_i S_{n - m - 1;i} (X^{-1}) X^{i\bar i} - \sum_i S_{n - m - 1;i} (X^{-1}) (X^{i\bar i})^2 (\chi_{i\bar i} + \tilde \chi_{i\bar i} + \rho_{i\bar i}) \\
	&\leq \sum_i S_{n - m - 1;i} (X^{-1}) X^{i\bar i} - (1 + \delta) \sum_i S_{n - m - 1;i} (X^{-1}) (X^{i\bar i})^2 \chi_{i\bar i} \\
%	&=
%	\sum_i S_{n - m - 1;i} (X^{-1}) X^{i\bar i}
%	- \left(1 + \frac{\delta}{2}\right) \sum_i S_{n - m - 1;i} (X^{-1}) (X^{i\bar i})^2 \chi_{i\bar i} \\
%	&\qquad - \sum_i S_{n - m - 1;1} (X^{-1}) X^{1\bar 1} - \frac{\delta}{2} \sum_i S_{n - m - 1;i} (X^{-1}) (X^{i\bar i})^2 \chi_{i\bar i}  \\
%	&\qquad +  S_{n - m - 1;1} (X^{-1}) X^{1\bar 1} \\
	&\leq - S_{n - m} (X^{-1}) + S_{n - m} (\tilde X^{-1}) - \frac{\delta}{2} \sum_i S_{n - m - 1;i} (X^{-1}) (X^{i\bar i})^2 \chi_{i\bar i} \\
	&\qquad +   S_{n - m - 1;1} (X^{-1}) X^{1\bar 1} \\
	&\leq - S_{n - m} (X^{-1}) + S_{n - m} (\tilde X^{-1}) - \frac{\delta \kappa}{2} \sum_i S_{n - m - 1;i} (X^{-1}) (X^{i\bar i})^2 \\
		&\qquad +  S_{n - m - 1;1} (X^{-1}) X^{1\bar 1} 
	,
\end{aligned}
\end{equation}
where
\begin{equation}
\tilde X := \left(1 + \frac{\delta}{2}\right) \chi + X_{1\bar 1} \sqrt{-1}  dz^1 \wedge d\bar z^1 .
\end{equation}
In this proof, we define
\begin{equation}
\hat \chi:= \frac{\delta}{4} \chi + X_{1\bar 1} \sqrt{-1} dz^1 \wedge d\bar z^1 .
\end{equation}
Since $\chi$ satisfies the boundary case of cone condition, we obtain that 
\begin{equation}
\label{inequality-3-36}
\begin{aligned}
	&\quad c \tilde X^n - \left(1 + \frac{\delta}{4}\right)^{n - m} n \tilde X^m \wedge \omega^{n - m} \\
	&\geq \left(1 + \frac{\delta}{4}\right)^n \left( c  \frac{\hat \chi^n}{(1 + \frac{\delta}{4})^n} + c \chi^n - n\chi^m \wedge \omega^{n - m} \right) \\
%	&= c \hat \chi^n +  \left(1 + \frac{\delta}{4}\right)^n \left(c \chi^n - n \chi^m \wedge \omega^{n - m} \right) \\
	&> 
	c n \frac{\delta^{n - 1}}{4^{n - 1}} X_{1\bar 1} \sqrt{-1} dz^1 \wedge d\bar z^1 \wedge \chi^{n - 1} -  \left(1 + \frac{\delta}{4}\right)^n n \chi^m \wedge \omega^{n - m} \\
	&\geq 
	c n \frac{\delta^{n - 1} \kappa^{n - 1}}{4^{n - 1}} X_{1\bar 1}  \sqrt{-1} dz^1 \wedge d\bar z^1 \wedge \omega^{n - 1} -  \left(1 + \frac{\delta}{4}\right)^n n \chi^m \wedge \omega^{n - m} \\
	&\geq 0
	,
\end{aligned}
\end{equation}
if $X_{1\bar 1} \geq \left(1 + \frac{\delta}{4}\right)^n \frac{4^{n - 1} n}{c \delta^{n - 1} \kappa^{n - 1}}\frac{S_m (\chi)}{C^m_n}$.
%, which is a constant. 
Then from \eqref{inequality-3-36},
%\begin{equation}
%\begin{aligned}
%	S_{n - m} (\tilde X^{-1}) < \frac{1}{\left(1 + \frac{\delta}{4}\right)^{n - m}}\frac{C^m_n c}{n} .
%\end{aligned}
%\end{equation}
\begin{equation}
\begin{aligned}
	&\quad S_{n - m} (\tilde X^{-1}) +   S_{n - m - 1;1} (X^{-1}) X^{1\bar 1} \\
	&< \frac{1}{(1 + \frac{\delta}{4})^{n - m}} \frac{C^m_n c}{n} + C^{n - m - 1}_{n - 1} \left(\frac{2 n (n - m + 1)}{(n - m) (1 + \delta) \kappa}\right)^{n - m - 1} X^{1\bar 1} \\
%	&=
%	\frac{1}{(1 + \frac{\delta}{4})^{n - m}} \frac{C^m_n c}{n} + \frac{C^n_m c}{n}  \frac{n - m}{c}  \left(\frac{2 n (n - m + 1)}{(n - m) (1 + \delta) \kappa}\right)^{n - m - 1} X^{1\bar 1} \\
%	&= 
%	\frac{C^m_n c}{n} \left(\frac{1}{(1 + \frac{\delta}{4})^{n - m}} +  \frac{n - m}{c}  \left(\frac{2 n (n - m + 1)}{(n - m) (1 + \delta) \kappa}\right)^{n - m - 1} X^{1\bar 1} \right) \\
	&< 
	\frac{C^m_n c}{n} \frac{1}{\left(1 + \frac{\delta}{4}\right)^{n - m -1} \left(1 + \frac{\delta}{8}\right)} ,
\end{aligned}
\end{equation}
if
\begin{equation}
\begin{aligned}
	X_{1\bar 1} &> \left(1 + \frac{\delta}{4}\right)^n \frac{4^{n - 1} n}{c \delta^{n - 1} \kappa^{n - 1}}\frac{S_m (\chi)}{C^m_n} \\
	&\qquad + \frac{n - m}{c} \left(\frac{2n (n - m + 1)}{(n - m) (1 + \delta) \kappa}\right)^{n - m - 1} \frac{8 \left(1 + \frac{\delta}{4}\right)^{n - m}}{\delta} 
	.
\end{aligned}
\end{equation}
The proof is complete.

\end{proof}

Applying Lemma~\ref{lemma-3-1} to \eqref{inequality-3-28},
\begin{equation}
\label{inequality-3-39}
\begin{aligned}
	0
	&\geq 
	  A  \left(\frac{C^m_n c}{n} -   S_{n - m} (X^{-1})\right) + \frac{1}{w} \sum_{i,l} S_{n - m - 1;l} (X^{-1}) (X^{l\bar l})^2 G_{l\bar li\bar i}   \\
		&\qquad - \frac{1}{w} \sum_{i,l} S_{n - m - 1;l} (X^{-1}) X^{l\bar l} R_{i\bar il\bar l}  + \frac{1}{w} \sum_{i,l} S_{n - m - 1;l} (X^{-1}) (X^{l\bar l})^2 R_{l\bar li\bar i} X_{i\bar i}  \\
		&\qquad  - A \left( \frac{C^m_n c}{n} -  S_{n - m} (X^{-1}) - \theta - \theta \sum_i S_{n - m - 1;i} (X^{-1}) (X^{i\bar i})^2 \right) \\
%		&= A \theta \left(1 +   \sum_i S_{n - m - 1;i} (X^{-1}) (X^{i\bar i})^2 \right)  + \frac{1}{w} \sum_{i,l} S_{n - m - 1;l} (X^{-1}) (X^{l\bar l})^2 G_{l\bar li\bar i} \\
%		&\qquad - \frac{1}{w} \sum_{i,l} S_{n - m - 1;l} (X^{-1}) X^{l\bar l} R_{i\bar il\bar l}  + \frac{1}{w} \sum_{i,l} S_{n - m - 1;l} (X^{-1}) (X^{l\bar l})^2 R_{l\bar li\bar i} X_{i\bar i}  \\
		&\geq A \theta \left(1 +   \sum_i S_{n - m - 1;i} (X^{-1}) (X^{i\bar i})^2 \right)  + \frac{\inf_{i,l} G_{l\bar li\bar i}}{w} \sum_{l} S_{n - m - 1;l} (X^{-1}) (X^{l\bar l})^2 \\
		&\qquad  - \frac{n \sup_{i,l} |R_{i\bar il\bar l}|}{w} \sum_{l} S_{n - m - 1;l} (X^{-1}) X^{l\bar l }  -  \sup_{i,l} |R_{i\bar il\bar l}| \sum_{l} S_{n - m - 1;l} (X^{-1}) (X^{l\bar l })^2 \\
		&\geq A \theta \left(1 +   \sum_i S_{n - m - 1;i} (X^{-1}) (X^{i\bar i})^2 \right)  + \frac{\inf_{i,l} G_{l\bar li\bar i}}{w} \sum_{i} S_{n - m - 1;i} (X^{-1}) (X^{i\bar i})^2 \\
				&\qquad  - (n+1)  \sup_{i,l} |R_{i\bar il\bar l}|   \sum_{i} S_{n - m - 1;i} (X^{-1}) (X^{i\bar i})^2  
		,
\end{aligned}
\end{equation}
when $w > N + 1$ at $(z_{max},t_{max})$.  We can pick coefficient $A$ satisfying
\begin{equation}
	A \theta > - \inf_{i,l} G_{l\bar li\bar i} + (n+1) \sup_{i,l} |R_{i\bar il\bar l}|, 
\end{equation}
and hence see a contradiction in \eqref{inequality-3-39}.

If function $\ln w - A (\varphi - \rho)$ reaches its maximal value in $M \times (0,t']$,
\begin{equation*}
	\ln w - A (\varphi - \rho) \leq \ln w - A (\varphi - \rho) \Big|_{(z_{max},t_{max})} 
	%\leq \ln (N + 1) - A \varphi (z_{max} , t_{max}) 
	\leq   \ln (N + 1) - A \min_M \varphi_0
	,
\end{equation*}
and hence we obtain
\begin{equation*}
	\ln w 
%	\leq 
%	\ln (N + 1) + A \left(\varphi - \varphi (z_{max},t_{max}) \right) - A \rho
	\leq
	\ln ( N + 1) + A \left(\varphi - \min_M \varphi_0 - \rho \right) 
%	\leq 
%	\ln (N + 1) + A \left(\max_M \varphi_0 + \Vert \psi\Vert_{L^\infty} - \min_M \varphi_0 \right) - A \rho 
%	= 
	\leq
	\ln (N + 1) + A \left(osc_M \varphi_0 + \Vert \psi\Vert_{L^\infty} - \rho  \right)  
	.
\end{equation*}
Otherwise, 
\begin{equation*}
	\ln w  - A (\varphi - \rho) \leq \ln w - A (\varphi _0 - \rho) \Big|_{z_{max}} 
%	\leq \ln (N + 1) - A \varphi_0 (z_{max} ) 
	\leq \ln (N + 1) - A \min_M \varphi_0,
\end{equation*}
and hence we again obtain
\begin{equation*}
	\ln w 
%	\leq 
%	\ln (N + 1) + A\left(\varphi - \varphi_0 (z_{max}) \right) - A \rho 
	\leq \ln (N + 1) + A \left(\varphi - \min_M \varphi_0 - \rho \right) 
%	\leq 
%	\ln (N + 1) + A\left(\max_M \varphi_0 + \Vert \psi \Vert_{L^\infty} - \min_M\varphi_0  \right) - A \rho 
%	=
	\leq 
	\ln (N + 1) + A\left(osc_M \varphi_0 + \Vert \psi \Vert_{L^\infty}  - \rho  \right) 
	.
\end{equation*}

\begin{theorem}

Let $\varphi \in C^4 (M \times [0,T))$ be an admissible solution to Equation~\eqref{equation-3-1}. Then there are uniform positive constants $C$ and $A$ such that
\(
	w \leq C e^{A (\varphi - \rho)} .
\)

\end{theorem}

\medskip
\subsection{Global second order estimate}

\begin{lemma}
\label{lemma-5-3}
Let $\varphi \in C^2(M)$ satisfy $\chi + \tilde \chi + \sqrt{-1} \partial\bar\partial \varphi > 0$. There are positive constants $N$ and $\theta$ such that when $w > N$ at a point $z \in M$,
\begin{equation}
\begin{aligned}
	&\quad \sum_i S_{n - m - 1;i} (X^{-1}) (X^{i\bar i})^2 \varphi_{i\bar i}  \\
	&\leq \frac{ 2 C^m_n c}{ n} -  S_{n - m} (X^{-1}) - \theta - \theta \sum_i S_{n - m - 1;i} (X^{-1}) (X^{i\bar i})^2 
	,
\end{aligned}
\end{equation}
under coordinates around $z$ such that $\omega_{i\bar  j} = \delta_{ij}$ and $X_{i\bar j}$ is diagonal at $z$.

\end{lemma}
\begin{proof}

Without loss of generality, we may assume that $\chi \geq \kappa \omega$ for $\kappa > 0$, and $X_{1\bar 1} \geq \cdots \geq X_{n\bar n}$. We calculate that
\begin{equation}
\label{inequality-5-31}
\begin{aligned}
	&\quad \sum_i S_{n - m - 1;i} (X^{-1}) (X^{i\bar i})^2 \varphi_{i\bar i}  \\
%	&=  \sum_i S_{n - m - 1;i} (X^{-1}) X^{i\bar i} - \sum_i S_{n - m - 1;i} (X^{-1}) (X^{i\bar i})^2 (\chi_{i\bar i} + \tilde \chi_{i\bar i} ) \\
%	&= (n - m) S_{n - m} (X^{-1}) -   \sum_i S_{n - m - 1;i} (X^{-1}) (X^{i\bar i})^2 (\chi_{i\bar i} + \tilde \chi_{i\bar i}) \\
	&\leq (n - m) S_{n - m} (X^{-1}) -  \kappa \sum_i S_{n - m - 1;i} (X^{-1}) (X^{i\bar i})^2 \\
%	&= (n - m) S_{n - m} (X^{-1}) - \frac{ \kappa}{2} \sum_i S_{n - m - 1;i} (X^{-1}) (X^{i\bar i})^2  - \frac{ \kappa}{2} \sum_i S_{n - m - 1;i} (X^{-1}) (X^{i\bar i})^2 \\
%	&= (n - m) S_{n - m} (X^{-1}) - \frac{ \kappa}{2} \sum_i \left( S_{n - m } (X^{-1}) - S_{n - m;i } (X^{-1}) \right) X^{i\bar i} \\
%	&\qquad - \frac{\kappa}{2} \sum_i S_{n - m - 1;i} (X^{-1}) (X^{i\bar i})^2 \\
%	&= (n - m) S_{n - m} (X^{-1}) - \frac{ \kappa}{2} \left( S_{n - m } (X^{-1}) S_1 (X^{-1}) - (n - m + 1) S_{n - m + 1} (X^{-1}) \right) \\
%		&\qquad - \frac{ \kappa}{2} \sum_i S_{n - m - 1;i} (X^{-1}) (X^{i\bar i})^2 \\
	&\leq 
	(n - m) S_{n - m} (X^{-1}) - \frac{ \kappa}{2} \frac{n - m}{n} S_{n - m} (X^{-1}) S_1 (X^{-1}) \\
			&\qquad - \frac{ \kappa}{2} \sum_i S_{n - m - 1;i} (X^{-1}) (X^{i\bar i})^2 \\
	&=
	(n - m) \left(1 - \frac{  \kappa}{2 n} S_1 (X^{-1})\right)S_{n - m} (X^{-1}) - \frac{  \kappa}{2} \sum_i S_{n - m - 1;i} (X^{-1}) (X^{i\bar i})^2
	.
\end{aligned}
\end{equation}
If $S_1 (X^{-1}) \geq \frac{2n (n - m + 1)}{(n - m)  \kappa}$, it is  from \eqref{inequality-5-31} that
\begin{equation*}
\begin{aligned}
%	&\quad 
	\sum_i S_{n - m - 1;i} (X^{-1}) (X^{i\bar i})^2 \varphi_{i\bar i} %\\
	&\leq - S_{n - m} (X^{-1}) - \frac{ \kappa}{2} \sum_i S_{n - m - 1;i} (X^{-1}) (X^{i\bar i})^2
	.
\end{aligned}
\end{equation*}
If $S_1 (X^{-1}) < \frac{2n (n - m + 1)}{(n - m)  \kappa}$, then
\(
	X_{1\bar 1} \geq \cdots \geq X_{n\bar n} > \frac{1}{S_1 (X^{-1})} > \frac{(n - m)   \kappa}{2n (n - m + 1)} 
\). 
So we have
\begin{equation*}
\begin{aligned}
	&\quad \sum_i S_{n - m - 1;i} (X^{-1}) (X^{i\bar i})^2 \varphi_{i\bar i}  \\
%	&= \sum_i S_{n - m - 1;i} (X^{-1}) X^{i\bar i} - \sum_i S_{n - m - 1;i} (X^{-1}) (X^{i\bar i})^2 (\chi_{i\bar i} + \tilde \chi_{i\bar i}  ) \\
%	&=
%	\sum_i S_{n - m - 1;i} (X^{-1}) X^{i\bar i}
%	- \left(1 - \frac{\delta}{2}\right) \sum_i S_{n - m - 1;i} (X^{-1}) (X^{i\bar i})^2 \chi_{i\bar i} \\
%	&\qquad - \sum_i S_{n - m - 1;1} (X^{-1}) X^{1\bar 1} - \frac{\delta}{2} \sum_i S_{n - m - 1;i} (X^{-1}) (X^{i\bar i})^2 \chi_{i\bar i}  \\
%	&\qquad +  S_{n - m - 1;1} (X^{-1}) X^{1\bar 1} \\
	&\leq - S_{n - m} (X^{-1}) + S_{n - m} (\tilde X^{-1}) - \frac{\delta}{2} \sum_i S_{n - m - 1;i} (X^{-1}) (X^{i\bar i})^2 \chi_{i\bar i} \\
	&\qquad +   S_{n - m - 1;1} (X^{-1}) X^{1\bar 1} \\
	&\leq - S_{n - m} (X^{-1}) + S_{n - m} (\tilde X^{-1}) - \frac{\delta \kappa}{2} \sum_i S_{n - m - 1;i} (X^{-1}) (X^{i\bar i})^2 \\
		&\qquad +  S_{n - m - 1;1} (X^{-1}) X^{1\bar 1} 
	,
\end{aligned}
\end{equation*}
where $\delta > 0$ is to be specified later and
\begin{equation}
\tilde X := \left(1 - \frac{\delta}{2}\right) \chi +  X_{1\bar 1} \sqrt{-1} dz^1 \wedge d\bar z^1 .
\end{equation}
In this proof, we denote
\begin{equation}
\hat \chi:= \frac{\delta}{2} \chi +  X_{1\bar 1}  \sqrt{-1}  dz^1 \wedge d\bar z^1 .
\end{equation}
Since $\chi$ satisfies the boundary case of cone condition, we obtain that 
\begin{equation}
\label{inequality-5-37}
\begin{aligned}
	&\quad c \tilde X^n - \left(1 - \delta \right)^{n - m} n \tilde X^m \wedge \omega^{n - m}	
%	\\
%	&= c \left( (1 - \delta ) \chi + \hat\chi\right)^n - (1 - \delta)^{n - m} n ((1 - \delta) \chi + \hat\chi)^m \wedge \omega^{n - m} \\
%	&= c (1 - \delta)^n \left(\chi + \frac{\hat\chi}{1 - \delta}\right)^n - (1 - \delta)^n n \left(\chi + \frac{\hat\chi}{1 - \delta}\right)^m \wedge \omega^{n - m}  \\
%	&
	\geq 
%	\left(1 - \delta\right)^n \left( c  \frac{\hat \chi^n}{(1 - \delta)^n} + c \chi^n - n\chi^m \wedge \omega^{n - m} \right) \\
%	&= 
	c \hat \chi^n +  \left(1 - \delta\right)^n \left(c \chi^n - n \chi^m \wedge \omega^{n - m} \right) 
%	\\
%	&
	> 
%	c n \frac{\delta^{n - 1}}{2^{n - 1}}  X_{1\bar 1} \sqrt{-1} dz^1 \wedge d\bar z^1 \wedge \chi^{n - 1} -  \left(1 - \delta\right)^n n \chi^m \wedge \omega^{n - m} \\
%	&\geq 
%	c n \frac{\delta^{n - 1} \kappa^{n - 1}}{2^{n - 1}} X_{1\bar 1}  \sqrt{-1}  dz^1 \wedge d\bar z^1 \wedge \omega^{n - 1} -  \left(1 - \delta\right)^n n \chi^m \wedge \omega^{n - m} \\
%	&\geq 
	0
	,
\end{aligned}
\end{equation}
if $X_{1\bar 1} \geq \left(1 - \delta\right)^n \frac{2^{n - 1} n}{c \delta^{n - 1} \kappa^{n - 1}}\frac{S_m (\chi)}{C^m_n}$. %, which is a constant.  
Then from \eqref{inequality-5-37},
\begin{equation}
\begin{aligned}
	&\quad S_{n - m} (\tilde X^{-1}) +   S_{n - m - 1;1} (X^{-1}) X^{1\bar 1} \\
	&< \frac{1}{(1 - \delta)^{n - m}} \frac{C^m_n c}{n} + C^{n - m - 1}_{n - 1} \left(\frac{2 n (n - m + 1)}{(n - m)  \kappa}\right)^{n - m - 1} X^{1\bar 1} \\
%	&=
%	\frac{1}{(1 - \delta)^{n - m}} \frac{C^m_n c}{n} + \frac{C^n_m c}{n}  \frac{n - m}{c}  \left(\frac{2 n (n - m + 1)}{(n - m)  \kappa}\right)^{n - m - 1} X^{1\bar 1} \\
%	&= 
%	\frac{C^m_n c}{n} \left(\frac{1}{(1 - \delta)^{n - m}} +  \frac{n - m}{c}  \left(\frac{2 n (n - m + 1)}{(n - m)   \kappa}\right)^{n - m - 1} X^{1\bar 1} \right) \\
	&< 
	\frac{C^m_n c}{n} \frac{1}{\left(1 - \delta\right)^{n - m + 1} } ,
\end{aligned}
\end{equation}
if
\begin{equation}
\begin{aligned}
	X_{1\bar 1} &> \left(1 - \delta\right)^n \frac{2^{n - 1} n}{c \delta^{n - 1} \kappa^{n - 1}}\frac{S_m (\chi)}{C^m_n} \\
	&\qquad + \frac{n - m}{c} \left(\frac{2n (n - m + 1)}{(n - m)  \kappa}\right)^{n - m - 1} \frac{ \left(1 - \delta \right)^{n - m + 1}}{\delta} 
	.
\end{aligned}
\end{equation}
Choosing $\delta > 0$ sufficiently small, the proof is complete.

\end{proof}

We shall consider the function
\begin{equation}
	\ln w -  A \left(\varphi + c t\right) ,
\end{equation}
where $A$ is to be specified later. 
Function $\ln w - A \left(\varphi + ct\right)$ can reach its maximal value on $M \times [0,t']$ for  a fixed time $0 < t' < T$.  
Suppose that the maximal value appears at point $(z_{max},t_{max}) \in M \times (0,t']$.
We are able to pick a local chart around $z_{max}$ such that $\omega_{i\bar j} = \delta_{ij}$ and $X_{i\bar j}$ is diagonal at $z_{max}$ when $t = t_{max}$.
Therefore, at $(z_{max},t_{max})$
\begin{equation}
\label{inequality-5-41}
	\frac{\partial_t w}{w} - A (\partial_t \varphi + c) \geq 0 ,
\end{equation}
\begin{equation}
\label{inequality-5-42}
	\frac{\partial_l w}{w} - A \varphi_l = 0 ,
\end{equation}
and
\begin{equation}
\label{inequality-5-43}
	\frac{\bar\partial_l\partial_l w}{w} - \frac{\bar\partial_l w \partial_l w}{w^2} - A \varphi_{l\bar l} \leq 0 .
\end{equation}
As the argument in \ref{uniform-C2}, we have that
%By Lemma~\ref{lemma-5-3} and \eqref{inequality-5-46}, we derive that
\begin{equation}
\label{inequality-5-47}
\begin{aligned}
	0 &\geq
	A \frac{C^m_n}{n }  ( \partial_t \varphi + c)   - \left(  (n + 1) \sup_{i,l} |R_{i\bar il\bar l}| - \frac{\inf_{i,l} G_{l\bar li\bar i}}{w} \right)\sum_{l}  S_{n - m - 1;l} (X^{-1}) (X^{l\bar l})^2  \\
	&\qquad  + A \left(S_{n - m } (X^{-1}) + \theta + \theta \sum_i S_{n - m - 1;i} (X^{-1}) (X^{i\bar i})^2 - \frac{2 C^m_n c}{n}\right) \\
	&=
- \left(  (n + 1) \sup_{i,l} |R_{i\bar il\bar l}| - \frac{\inf_{i,l} G_{l\bar li\bar i}}{w} \right)\sum_{i}  S_{n - m - 1;i} (X^{-1}) (X^{i\bar i})^2  \\
	&\qquad + A \theta \left(1 + \sum_i S_{n -m - 1;i} (X^{-1}) (X^{i\bar i})^2\right)
	,
\end{aligned}
\end{equation}
when $w > N + 1$ at $(z_{max}, t_{max})$. Choosing $A > \frac{1}{\theta}\left(  (n + 1) \sup_{i,l} |R_{i\bar il\bar l}| - \inf_{i,l} G_{l\bar li\bar i}   \right)$, we find a contradiction in  \eqref{inequality-5-47}.

If function $\ln w - A (\varphi + ct)$ reaches its maximal value in $M \times (0,t']$,
\begin{equation*}
	\ln w - A (\varphi + c t) \leq \ln w - A (\varphi + ct) \Big|_{(z_{max},t_{max})} \leq \ln (N + 1) - A \varphi (z_{max} , t_{max}) ,
\end{equation*}
and hence we obtain
\begin{equation*}
	\ln w 
%	\leq 
%	\ln (N + 1) + A \left(osc_M \varphi_0 + \Vert \psi\Vert_{L^\infty}   \right) - A c(t_{max} - t)
	\leq 
	\ln (N + 1) + A (\varphi - \min_M \varphi_0 + ct)
	\leq 
	\ln (N + 1) + A \left(osc_M \varphi_0 + \Vert \psi\Vert_{L^\infty}  + ct \right) 
	.
\end{equation*}
Otherwise, 
\begin{equation*}
	\ln w  - A (\varphi + ct) \leq \ln w - A \varphi _0  \Big|_{z_{max}} \leq \ln (N + 1) - A \varphi_0 (z_{max} ) ,
\end{equation*}
and hence we again obtain
\begin{equation*}
	\ln w 
		\leq 
		\ln (N + 1) + A (\varphi - \min_M \varphi_0 + ct)
	\leq 
	\ln (N + 1) + A\left(osc_M \varphi_0 + \Vert \psi \Vert_{L^\infty}  + ct  \right)  
	.
\end{equation*}

\begin{theorem}
\label{theorem-5-4}
Let $\varphi \in C^4 (M \times [0,T))$ be an admissible solution to Equation~\eqref{equation-3-1}. Then there are uniform positive constants $C$ and $A$ such that
\(
	w \leq C e^{A(\varphi + c t)} .
\)

\end{theorem}

\medskip

\section{Long time existence and convergence}

\medskip

\subsection{Higher order estimates}

It is a routine step to obtain higher order estimates through Evans-Krylov theory  and Schauder estimate. Observing the $C^0$ bound is $t$-independent and higher order estimates are local,  we obtain both uniform smoothness outside $Amp (\tilde \chi)$ and global smoothness depending on time $t$. For more details, we refer readers to \cite{Evans1982}\cite{Krylov1982}\cite{Trudinger1983}\cite{Wang1992-1}\cite{Wang1992-2}.

\medskip

\subsection{Long time existence}

When $t = 0$,  the term $- \frac{n (\chi + \tilde + \sqrt{-1} \partial\bar\partial \varphi_0)^m \wedge \omega^{n - m}}{(\chi + \tilde \chi + \sqrt{-1} \partial\bar\partial \varphi_0)^n}$ is a non-degenerate nonlinear elliptic operator. 
By standard parabolic theory, there is a constant $\epsilon_1 > 0$ so that there is a unique smooth admissible solution $\varphi$ to Equation~\eqref{equation-3-1} when $0 \leq t < \epsilon_1$. 
Assume that $\varphi$ exists on the maximal time interval $[0,T)$, where $\epsilon_1 \leq T \leq + \infty$.
It suffices to show that $T = +\infty$.

Suppose that $T < +\infty$. By Theorem~\ref{theorem-5-4} and the $C^0$ bound, we know that there are positive constants $C$ and $A$ so that
\begin{equation}
	w \leq C e^{A c T} < +\infty
\end{equation}
on the maximal time interval $[0,T)$. 
Then we can obtain higher order estimates of $\varphi$ on $[0,T)$ through Evans-Krylov theory and Schauder estimate.
By the short time existence, there is a constant $\epsilon_2 > 0$ so that $\varphi$ can be smoothly extended to $[0,M+\epsilon_2)$, which is a contradiction.

\medskip

\subsection{Convergence}

We shall adapt the argument in \cite{FangLaiSongWeinkove2014} to show that the solution flow converges to a pluripotential solution to Equation~\eqref{elliptic-equation}.
%
%From \eqref{inequality-2-10}
%
%
%
%
%
Since $M \setminus Amp(\tilde\chi)$ is  in a proper analytic subset, $M\setminus Amp(\tilde \chi)$ is pluripolar and hence has zero capacity.

Since  all local bounds of derivatives of $\varphi$  in $Amp(\tilde\chi)$ are available, we can find a time sequence $\{t_i\}$ with $t_i \to +\infty$ such that $\varphi (z,t_i) \to \varphi_\infty (z)$ in $C^\infty_{loc} (Amp (\tilde \chi))$. For $z \not\in Amp(\tilde \chi)$, we define
\begin{equation}
\label{definition-5-2}
	\varphi_\infty(z) := \varphi^*_\infty (z) ,
\end{equation}
which is $(\chi + \tilde \chi)-PSH$ by Hartogs lemma and localization. 
%$\chi+\tilde\chi-PSH$

Suppose that there is a point $z_0 \in Amp (\tilde \chi)$ such that
$	\lim_{t \to \infty} \partial_t \varphi (z_0 , t) \neq 0 $.
Then there exists a sequence of time points $\{t_i\}$ such that $\lim_{i\to +\infty}t_i = \infty $ and $\left|\partial_t \varphi\right| > \epsilon$ $\forall i$.
Since we have all local bounds of derivatives of $\varphi$  in $Amp(\tilde\chi)$, there is a neighborhood $U$ around $z_0$ and a constant $\delta > 0$ such that
\begin{equation*}
\left|\frac{\partial\varphi}{\partial t}\right| > \frac{\epsilon}{2}, \qquad \forall z \in Amp (\tilde \chi), t\in [t_i,t_i + \delta].
\end{equation*}
Without loss of generality, we may assume that $t_{i + 1} - t_i > \delta$ and then obtain that
\begin{equation*}
\begin{aligned}
	&\quad \int^{+\infty}_0 \left(\int_M \left( \frac{\partial \varphi}{\partial t} \right)^2 (\chi + \tilde \chi + \sqrt{-1} \partial\bar\partial \varphi)^n \right) dt \\
	&\geq \sum^\infty_{i = 1} \int^{t_i + \delta}_{t_i} \left(\int_M \frac{\epsilon^2}{4} (\chi + \tilde \chi + \sqrt{-1}\partial\bar\partial \varphi)^n \right) dt \\
	&= \frac{\epsilon^2 \delta}{4} \sum^{\infty}_{i = 1} \int_M (\chi + \tilde \chi)^n \\
	&= +\infty , 
\end{aligned}
\end{equation*}
which contradicts \eqref{inequality-2-10}. 
It has to hold true that 
\begin{equation}
\label{inequality-5-2}
\lim_{t \to +\infty} \partial_t \varphi (z,t) = 0 \qquad \text{ for all }z \in Amp (\tilde \chi) .
\end{equation}
Suppose that there are $\epsilon > 0$ and a time sequence $\{t_i\}$ with $t_i \to +\infty$ such that
\begin{equation*}
	\Vert \partial_t \varphi (z,t_i)\Vert_{C^k (K)} > \epsilon ,
\end{equation*}
for some $k \in \mathbb{N}$ and compact subset $K \in Amp(\tilde \chi)$. 
Since we have all local bounds of $\partial_t \varphi$ in $Amp (\tilde \chi)$, $\partial_t \varphi (z,t_i)$ is convergent to a nonzero function, which contradicts \eqref{inequality-5-2}. 
We know that $\partial_t\varphi \to 0 $ in $C^\infty_{loc} (Amp(\tilde \chi))$, 
and obtain that
\begin{equation*}
	c(\chi + \tilde \chi + \sqrt{-1} \partial\bar\partial \varphi_\infty)^n = n (\chi + \tilde \chi + \sqrt{-1}\partial\bar\partial \varphi_\infty)^m \wedge \omega^{n - m} \qquad \text{on } Amp (\tilde \chi).
\end{equation*}
By continuity of Complex Monge-Amp\`ere operator and localization,
\begin{equation*}
	c(\chi + \tilde \chi + \sqrt{-1} \partial\bar\partial \varphi_\infty)^n = n (\chi + \tilde \chi + \sqrt{-1}\partial\bar\partial \varphi_\infty)^m \wedge \omega^{n - m} ,
\end{equation*}
in pluripotential sense with
\begin{equation}
\label{equality-5-8}
J_n (\chi + \tilde \chi, \varphi_\infty) = J_n (\chi + \tilde \chi, \varphi_0) .
\end{equation}

Suppose that there exist $\epsilon > 0$ and a time sequence $\{t_i\}$ with $t_i \to + \infty$ such that
\begin{equation*}
	\Vert \varphi(z,t_i) - \varphi_\infty (z)\Vert_{C^k (K)} > \epsilon ,
\end{equation*}
for some $k\in\mathbb{N}$ and compact subset $K\in Amp (\tilde\chi)$. By passing to a subsequence, $\varphi (z,t_i)$ converges to a function $\varphi'_\infty \neq \varphi_\infty$ by Arzel\`a-Ascoli theorem. According to the previous argument, 
\begin{equation*}
	c(\chi + \tilde \chi + \sqrt{-1} \partial\bar\partial \varphi'_\infty)^n = n (\chi + \tilde \chi + \sqrt{-1}\partial\bar\partial \varphi'_\infty)^m \wedge \omega^{n - m} ,
\end{equation*}
in pluripotential sense with
\begin{equation}
\label{equality-5-10}
J_n (\chi + \tilde \chi, \varphi'_\infty) = J_n (\chi + \tilde \chi, \varphi_0) .
\end{equation}
By the uniqueness result in \cite{Sun202210}, $\varphi_\infty - \varphi'_\infty$ is constant in $Amp (\tilde \chi)$. 
Then $\varphi_\infty - \varphi'_\infty$ is constant on $M$ accroding to Definition~\eqref{definition-5-2}.
Comparing \eqref{equality-5-8} and \eqref{equality-5-10}, $\varphi_\infty \equiv \varphi'_\infty$, which is a contradiction!

Further, we can conclude that for any $\varphi_0 \in \mathcal{H}$,
\begin{equation}
\label{inequality-5-6}
c J_n (\chi + \tilde \chi, \varphi_0) - n J_m (\chi + \tilde \chi, \varphi_0) \leq c J_n (\chi + \tilde \chi, \varphi_\infty) - n J_m (\chi + \tilde \chi, \varphi_\infty)
\end{equation}
by the uniqueness up to a constant additive and \eqref{equality-2-9}. The right-side term in \eqref{inequality-5-6} is a uniform bound, and hence Inequality~\eqref{inequality-5-6} can help to characterize the K\"ahler classes with proper Mabuchi energy as in \cite{SongWeinkove2008}.

%
%
%\newpage
%
%
%
%%
%%Inequality~\eqref{inequality-2-5} tells us that $\partial_t \varphi$ is uniformly bounded.
%%Then we have
%%\begin{equation*}
%%\lim_{i\to \infty}\int_M | \partial_t \varphi(z,t_i) |^2 \omega^n = 0
%%\end{equation*}
%%for any time sequence $\{t_i\}$ with $t_i \to +\infty$. 
%%So we derive that
%%\begin{equation}
%%	\lim_{t\to+\infty} \int_M | \partial_t \varphi |^2 \omega^n = 0 .
%%\end{equation}
%
%
%
%
%
%
%
%
%
%
%
%
%
%
%
%
%
%
%
%
%
%
%
%
%
%
%
%Applying Newton-Maclaurin inequality to \eqref{inequality-2-6}, we have
%\begin{equation}
%\label{inequality-5-4}
%	\begin{aligned}
%		\frac{(\chi + \tilde \chi + \sqrt{-1} \partial\bar\partial \varphi)^n}{\omega^n} &\geq 
%		\left(\frac{(\chi + \tilde \chi + \sqrt{-1} \partial\bar\partial \varphi)^n}{(\chi + \tilde \chi + \sqrt{-1} \partial\bar\partial \varphi)^m \wedge \omega^{n - m}}\right)^{\frac{n}{n - m}} \\
%		&\geq 
%		\left(\frac{(\chi + \tilde \chi + \sqrt{-1} \partial\bar\partial \varphi_0)^n}{(\chi + \tilde \chi + \sqrt{-1} \partial\bar\partial \varphi_0)^m \wedge \omega^{n - m}}\right)^{\frac{n}{n - m}}		
%		&\geq c_1  ,
%	\end{aligned}
%\end{equation}
%for some constant $c_1 > 0$.
%Substituting \eqref{inequality-5-2} into \eqref{inequality-2-10},
%\begin{equation}
%	\int^{+\infty}_0 \left(\int_M \left(\frac{\partial \varphi}{\partial t}\right)^2 \omega^n \right) dt \leq C,
%\end{equation}
%where $C$ depends on the uniform estimate of $\varphi$. 

\medskip
\noindent
{\bf Acknowledgements}\quad
The author wish to thank  Chengjian Yao and Ziyu Zhang for their helpful discussions and suggestions. The  author is supported by a start-up grant from ShanghaiTech University.

%\newpage


\medskip
\begin{thebibliography}{999}

\bibitem{Boucksom2004}
S. Boucksom,
{\em Divisorial Zariski decompositions on compact complex manifolds},
Amm. Sci. Ecole Norm. Sup., {\bf 37} (2004), no. 4, 45--76.






\bibitem{Chen2000}
X.-X. Chen,
{\em On the lower bound of the Mabuchi energy and its application},
Int. Math. Res. Notices {\bf 12} (2000), 607--623.




\bibitem{Chen2004}
X.-X. Chen,
{\em A new parabolic flow in K\"ahler manifolds}, 
Comm. Anal. Geom. {\bf 12} (2004), no. 4, 837--852.




\bibitem{Donaldson1999}
S. K. Donaldson,
{\em Moment maps and diffeomorphisms},
Asian J. Math. {\bf 3} (1999), 1--15.



\bibitem{Evans1982}
L. C. Evans,
{\em Classical solutions of  fully nonlinear, convex, second order elliptic equations},
Comm. Pure Appl. Math. {\bf 35} (1982), 333--363.





\bibitem{FangLaiMa2011}
H. Fang, M.-J. Lai and X.-N. Ma,
{\em On a class of fully nonlinear flows in K\"ahler geometry},
J. Reine Angew. Math. {\bf 653} (2011), 189--220.



\bibitem{FangLaiSongWeinkove2014}
H. Fang, M.-J. Lai, J. Song and B. Weinkove,
{\em The $J$-flow on K\"ahler surfaces: a boundary case}, 
Anal. PDE, {\bf 7} (2014), no. 1, 215--226. 




\bibitem{Krylov1982}
N. V. Krylov,
{\em Boundedly nonhomogeneous elliptic and parabolic equations},
Izvestiya Ross. Akad. Nauk. SSSR {\bf 46} (1982), 487--523.





\bibitem{EyssidieuxGuedjZeriahi2009}
P. Eyssidieux, V. Guedj and A. Zeriahi,
{\em Singular K\"ahler-Einstein metrics},
J. Amer. Math. Soc. {\bf 22} (2009), 607--639.





\bibitem{SongWeinkove2008}
J. Song and B. Weinkove, 
{\em On the convergence and singularities of the $J$-flow with applications to the Mabuchi energy}, 
Comm. Pure Appl. Math. {\bf 61} (2008), 210--229  



\bibitem{Sun2015p}
W. Sun,
{\em Parabolic complex Monge-Amp\`ere type equations on closed Hermitian manifolds}, 
Calc. Var. PDE {\bf 54} (2015), 3715--3733.



\bibitem{Sun202210}
W. Sun, 
{\em The boundary case for complex Monge-Amp\`ere type equations},
preprint.






\bibitem{Trudinger1983}
N. S. Trudinger, 
{\em Fully nonlinear, uniformly elliptic equations under natural structure conditions},
Trans. Am. Math. Soc. {\bf 278} (1983), no. 2, 751--769.







\bibitem{Wang1992-1}
L.-H. Wang,
{\em On the regularity theory of fully nonlinear parabolic equations. I},
Comm. Pure Appl. Math. {\bf 45} (1992), no. 1, 27--76.







\bibitem{Wang1992-2}
L.-H. Wang,
{\em On the regularity theory of fully nonlinear parabolic equations. II},
Comm. Pure Appl. Math. {\bf 45} (1992), no. 2, 141--178.





\bibitem{Yau1978}
S.-T. Yau,
{\em On the Ricci curvature of a compact K\"ahler manifold and the complex Monge-Amp\`ere equation. I.}, 
Comm. Pure Appl. Math. {\bf 31} (1978), no. 3, 339--411.




\end{thebibliography}
\end{document}